\RequirePackage[english]{babel}
\documentclass[a4paper,twoside,11pt]{amsart}
\usepackage[latin1]{inputenc}
\usepackage{amsmath}
\usepackage{amssymb}
\usepackage{amsfonts}
\usepackage{mathrsfs}
\usepackage{hyperref}
\usepackage{mathtools}
\usepackage{fixltx2e}[2005/12/01]

\usepackage{enumerate}
\renewcommand{\theenumi}{{\upshape{(\roman{enumi})}}}

\let\itemref\ref

\ifpdf
  \usepackage[matrix,arrow,cmtip]{xy}
\else
  \usepackage[matrix,arrow,cmtip,dvips]{xy}
\fi
\SelectTips{cm}{}
\newdir{(}{{}*!/-5pt/@^{(}}
\newdir{(x}{{}*!/-5pt/@_{(}}
\newdir{+}{{}*!/-9pt/{}}
\newdir{>+}{@{>}*!/-9pt/{}}
\entrymodifiers={+!!<0pt,\fontdimen22\textfont2>}

\usepackage[nosubsections]{mytheorems2}

\newtheoremstyle{dtheoremnopar}{3 mm}{1 mm}{\itshape}{}{\bfseries}{.}{ }
{\thmname{#1}\thmnumber{ #2}\thmnote{ \mdseries(#3)\bfseries}}

\theoremstyle{dtheoremnopar}
\newcounter{theoremx}

\newtheorem{conjecturealpha}[theoremx]{Conjecture}

\newcommand{\tref}[1]{\ref{#1}} 

\newcommand{\pref}[1]{\eqref{#1}}

\newcommand\Z{\mathbb{Z}}
\newcommand\Q{\mathbb{Q}}
\newcommand\inj{\hookrightarrow}
\newcommand\surj{\twoheadrightarrow}
\newcommand\map[3]{#1\colon #2\rightarrow #3}
\DeclareMathOperator{\rank}{rk}
\DeclareMathOperator{\image}{im}
\newcommand\sE{\mathcal{E}}
\newcommand\sF{\mathcal{F}}
\newcommand\sL{\mathcal{L}}
\newcommand\sG{\mathcal{G}}
\newcommand\sO{\mathcal{O}}
\DeclareMathOperator{\Ass}{Ass} 
\renewcommand\P{\mathbb{P}}           
\DeclareMathOperator{\Spec}{Spec}
\newcommand\red{\mathrm{red}}   
\newcommand\GL{\mathrm{GL}}
\newcommand{\etale}{\'{e}tale}
\newcommand\QCoh{\mathbf{QCoh}} 

\newcommand{\Irr}{\mathrm{Irr}}
\newcommand{\devissage}{d\'evissage}
\newcommand{\loccit}{\emph{loc.\ cit.}}

\newcommand{\spref}[1]{\href{http://stacks.math.columbia.edu/tag/#1}{#1}}

\begin{document}

\title[Approximation of sheaves]{Approximation of sheaves on algebraic stacks}
\author{David Rydh}
\address{KTH Royal Institute of Technology\\Department of 
  Mathematics\\SE-100 44 Stockholm\\Sweden}
\email{dary@math.kth.se}
\date{2015-05-07}
\thanks{Supported by the Swedish Research Council grant no 2011-5599.}

\begin{abstract}
Raynaud--Gruson characterized flat and pure morphisms between
affine schemes in terms of projective modules. We give a similar
characterization for non-affine morphisms. As an application,
we show that every quasi-coherent sheaf is the union of its finitely
generated quasi-coherent subsheaves on any quasi-compact and quasi-separated
algebraic stack.
\end{abstract}

\maketitle


\begin{section}{Introduction}
It is well-known that on a noetherian scheme every quasi-coherent sheaf is the
union of its coherent subsheaves~\cite[Cor.~9.4.9]{egaI}. This is also
true for noetherian algebraic stacks~\cite[Prop.~15.4]{laumon}.

For a non-noetherian scheme or algebraic stack $X$, this question splits up
into two questions.
\begin{enumerate}
\item Is every quasi-coherent $\sO_X$-module the union of its
  quasi-coherent submodules of finite type?
\item Is every quasi-coherent $\sO_X$-module a directed colimit of finitely
  presented $\sO_X$-modules?
\end{enumerate}
When these questions have positive answers, we say that $X$ has the
\emph{partial completeness property} and \emph{completeness property}
respectively. The second property implies the first (take images).

It is known that quasi-compact and quasi-separated schemes have the
completeness property~\cite[\S6.9]{egaI_NE}.
In~\cite[Thm.~A]{rydh_noetherian-approx} it was shown
that many stacks, including quasi-compact and quasi-separated algebraic spaces
and Deligne--Mumford stacks, have the completeness property. With current
technology, this result only applies to relatively few algebraic stacks with
infinite stabilizer groups.

The main result of this paper settles the partial completeness property for
every reasonable stack.
\begin{theorem*}
Let $X$ be a quasi-compact and quasi-separated algebraic stack. Then
every quasi-coherent $\sO_X$-module is the union of its quasi-coherent
submodules of finite type.
\end{theorem*}

An important application of the theorem is that when $X$ in addition
has affine stabilizer groups, then there exists a finitely
presented filtration of $X$ with strata that are global quotient
stacks~\cite[Prop.~2.6~(i)]{hall-rydh_alg-groups-classifying}. This is used to
obtain a criterion for an algebraic stack to have finite cohomological
dimension~\cite[Thm.~2.1]{hall-rydh_alg-groups-classifying} and to extend
Tannaka duality to non-noetherian stacks~\cite[Thm.~1.4]{hall-rydh_coherent-tannaka-duality}.

The key idea in the proof of the main theorem
is to use projective modules instead of flat modules. The main
lemma~\pref{L:affine-projective-submodule-descent} on existence of minimal
modules goes back to Serre~\cite[Exp.~VIB, 11.8, 11.10.1]{SGA3} in the context
of coalgebras and comodules. Here projectivity cannot be replaced with
flatness. The bulk of the paper extends this result to non-affine
\emph{pure} morphisms (Theorem~\ref{T:existence-of-minimal}). For this,
we use a new characterization of pure morphisms between stacks in terms
of projectivity (Theorem~\ref{T:pure=hom-proj}). This generalizes the
characterization of affine pure morphisms due to
Raynaud--Gruson~\cite[Thm.~I.3.3.5]{raynaud-gruson}.

The main result naturally leads to the following conjectures.
\begin{conjecturealpha}\label{CONJ:completeness}
  If $X$ is a quasi-compact and quasi-separated algebraic stack,
  then $X$ has the completeness property.
\end{conjecturealpha}
\begin{conjecturealpha}\label{CONJ:approximation}
  If $X$ is a quasi-compact and quasi-separated algebraic stack,
  then $X$ has an approximation, that is, there exists a factorization
  $X\to X_0\to \Spec \Z$ where $X\to X_0$ is affine and $X_0$ is of
  finite presentation over $\Spec \Z$.
\end{conjecturealpha}
The second conjecture implies the first conjecture. Our proof of the main
theorem first reduces the question to the case when there is a pure
presentation. The conjectures can also be reduced to this seemingly simpler
situation
(see Remark~\ref{R:conjectures}).

In Sections~\ref{S:proj}--\ref{S:pure}, we recall and extend some notions from
schemes to algebraic stacks. This includes (1) locally free and locally
projective modules, (2) assassins and schematically dominant morphisms, and (3)
pure morphisms.
In Section~\ref{S:hom-proj}, we give a characterization
of pure morphisms in terms of projectivity (Theorem~\ref{T:pure=hom-proj}).
In Section~\ref{S:minimal}, we prove the existence of minimal subsheaves
for pure morphisms. In Section~\ref{S:approx}, we prove the
main theorem. In the last section, we give some applications to the
main theorem.

We follow the terminology of~\cite{stacks-project} and do not impose any
separation conditions on a general algebraic stack. An algebraic stack is
\emph{quasi-separated} if its diagonal is quasi-compact and quasi-separated,
that is, if the diagonal and the double diagonal are quasi-compact.

\begin{subsection}{Acknowledgments}
It is my pleasure to acknowledge useful discussions with Jack Hall and
useful comments from Martin Brandenburg and Matthieu Romagny.
I would also like to express my gratitude to the referees for their many useful
suggestions and corrections that improved the paper.
\end{subsection}
\end{section}


\begin{section}{Locally free and locally projective modules}\label{S:proj}
In this section, we recall some standard results on infinitely generated
projective modules due to Kaplansky, Bass and Raynaud--Gruson.

\begin{definition}
Let $X$ be an algebraic stack. We say that a quasi-coherent sheaf $\sF$ is
\emph{locally free} (resp.\ \emph{locally projective}) if there exists a
jointly surjective family of flat morphisms $\map{p_i}{\Spec A_i}{X}$, locally
of finite presentation, such that $p_i^*\sF$ is free (resp.\ projective) for
every $i$.
\end{definition}

We do not require that $p_i^*\sF$ has finite rank in the definition of
locally free.
Note that the properties locally free and locally projective are stable under
arbitrary pull-back and are local for
the fppf-topology. We have the implications: locally free $\implies$ locally
projective $\implies$ flat.

If $x\in |X|$ is a point, then we define the rank $\rank_\sF(x)$ of $\sF$
at $x$ as the
cardinality of a basis of the $k$-vector space $\varphi^*\sF$ for any
representative
$\map{\varphi}{\Spec k}{X}$ of $x$. Since flat morphisms that are
locally of finite presentation are open, the rank of $\sF$ is locally constant
on $|X|$ if $\sF$ is locally free.

The rank does not behave so well for flat modules that are not finitely
generated. If $A=\Z$ and $M=\Q$, then the rank of $M$ is not upper
semicontinuous. The rank of projective modules is more well-behaved.

\begin{lemma}[{Kaplansky~\cite{kaplansky_projective-modules}}]\label{L:kaplansky-local-ring}
If $A$ is a local ring, then every projective $A$-module is free.
\end{lemma}

Thus, if $X$ is a quasi-separated\footnote{This condition is necessary with
the naive notion of irreducible components, cf.\ Example~\pref{E:disconnected}.}
algebraic stack and $\sF$ is a locally
projective
$\sO_X$-module, then
\begin{enumerate}
\item the rank of $\sF$ is constant on irreducible components of $X$; and
\item if $X$ has a finite number of irreducible components
  (e.g., $X$ noetherian), then the rank is
  locally constant.
\end{enumerate}
Nevertheless, even if $M$ is projective and has finite rank at every point,
the rank need not
be locally constant. Bass gives an example, due to Kaplansky, of a projective
module of rank $\leq 1$ such that the locus where the module has rank~$0$ is
closed but not open~\cite[p.~31, (2)]{bass_big-projective}. We now give a
similar example.
\begin{example}\label{E:disconnected}
Let $k$ be an algebraically closed field and let $A=T(k[x])$ be the absolutely
flat ring associated to the polynomial ring
$k[x]$~\cite[Prop.~5]{olivier_sem_samuel}. Then $\Spec A$ is zero-dimensional
and reduced and its underlying topological space is the one-point
compactification of $k$ with its discrete topology. For every $\lambda \in k$,
the corresponding quotient $A\surj \kappa(\lambda)=k$ is a locally free and
finitely generated $A$-module, hence projective. The direct sum
$M=\oplus_{\lambda\in k} \kappa(\lambda)$ is a projective $A$-module with rank
$1$ over the open subset $k$ and rank $0$ over its complement, which consists of
a single point $\xi$.

The discrete additive group $G=(k,+)$ acts freely on $\Spec A$ and the quotient
$X=\Spec A/G$ is an algebraic space consisting of two points $\{x,\xi\}$ where
$x$ is open and $\xi$ is closed. Note that $X$ is not quasi-separated since the
orbit of $x$ is not quasi-compact. The module $M$ descends to a locally
projective $\sO_X$-module $\sF$ such that the rank over $x$ is one and the rank
over $\xi$ is zero. The topological space $|X|$ is irreducible and hence the
rank is not constant over irreducible components in the usual sense.
\end{example}

A flat module
that has constant rank need not be so nice either as the following example
shows.

\begin{example}
If $M\subseteq \Q$ is the $\Z$-submodule generated
by all $p^{-1}$, for prime numbers $p$, then $M$ is flat of constant rank $1$
but neither projective nor finitely generated.
\end{example}

\begin{proposition}\label{P:loc-proj-vs-loc-free}
Let $X$ be an algebraic stack and let $\sF$ be a quasi-coherent sheaf on $X$.
\begin{enumerate}
\item\label{PI:affine:proj=loc-proj}
  If $X$ is an affine scheme, then $\sF$ is locally projective if and only
  if $\sF$ is projective.
\item\label{PI:affine:proj=free:noeth}
  If $X$ is a noetherian affine scheme and $\aleph\geq \aleph_0$ is an infinite
  cardinal, then $\sF$ is projective with constant rank $\aleph$ if and
  only if $\sF$ is free of rank $\aleph$.
\item\label{PI:proj+finite-rank=vector-bundle:noeth}
  If $X$ is noetherian, then $\sF$ is locally projective of finite rank if and
  only if $\sF$ is finitely generated and locally free.
\item\label{PI:loc-free=loc-proj:noeth}
  If $X$ is noetherian, then $\sF$ is locally projective if and only if $\sF$ is
  locally free.
\item\label{PI:loc-free=Zar-loc-free:noeth-scheme}
  If $X$ is a noetherian scheme, then $\sF$ is locally free if and only if $\sF$
is Zariski-locally free.
\end{enumerate}
\end{proposition}
\begin{proof}
In each case, the ``if'' part is trivial.
The necessity of the first condition
follows from~\cite[I.3.1.4]{raynaud-gruson} (countable rank)
or~\cite[II.2.5.1 and II.3.1.3]{raynaud-gruson} (general case).
That conditions~\itemref{PI:affine:proj=free:noeth}
and~\itemref{PI:proj+finite-rank=vector-bundle:noeth} are necessary
is~\cite[Cor.~3.2 \& Prop.~4.2]{bass_big-projective} respectively.
Since the rank of a locally projective sheaf is locally constant on a
noetherian stack, the necessity of conditions~\itemref{PI:loc-free=loc-proj:noeth}
and~\itemref{PI:loc-free=Zar-loc-free:noeth-scheme} follow from
\itemref{PI:affine:proj=loc-proj}, \itemref{PI:affine:proj=free:noeth}
and~\itemref{PI:proj+finite-rank=vector-bundle:noeth}.
\end{proof}

\begin{remark}
Without the noetherian assumptions,
statements~\itemref{PI:proj+finite-rank=vector-bundle:noeth}
and~\itemref{PI:loc-free=loc-proj:noeth} are false. If
statement~\itemref{PI:affine:proj=free:noeth} holds without the noetherian
assumption, then so does~\itemref{PI:loc-free=Zar-loc-free:noeth-scheme}. In
particular, this would imply that on any stack $X$, a quasi-coherent sheaf
$\sF$ is locally free if and only if
$\sF$ is locally projective, has locally constant rank and
is finitely generated over the open locus of finite rank.
\end{remark}

\end{section}


\begin{section}{Relative assassins and relative faithfulness}
In this section, we extend the notions of relative
assassins~\cite[3.2.2]{raynaud-gruson} and schematically dominant
morphisms~\cite[11.9--11.10]{egaIV} from schemes to algebraic stacks.

\begin{xpar}[Associated points]\label{X:ass}
There is a unique notion of associated points of coherent
sheaves on locally noetherian algebraic stacks such that
\begin{enumerate}
\item it coincides with the usual one for schemes; and
\item\label{XI:flat}
  if $\map{f}{X}{Y}$ is a flat morphism between locally noetherian
  stacks and $\sF$ is a coherent $\sO_Y$-module, then
  $f(\Ass_X(f^*\sF))\subseteq \Ass_Y(\sF)$ with equality if $f$ is
  surjective.
\end{enumerate}
The usual assassin satisfies (ii) for morphisms between schemes.
Indeed, more precisely we have that
\begin{equation}\label{E:Ass}
\Ass_X(f^*\sF)=\bigcup_{\mathclap{y\in \Ass_Y(\sF)}} \Ass_{X_y}(\sO_{X_y})
\end{equation}
for any flat morphism $\map{f}{X}{Y}$ between
locally noetherian schemes~\cite[Prop.~3.3.1]{egaIV}.
We may thus simply
define $\Ass_{X}(\sF)$ for a coherent sheaf $\sF$ on $X$
as $\Ass_{X}(\sF):=p(\Ass_U(p^*\sF))$ where $\map{p}{U}{X}$ is a
presentation. One can also give a more intrinsic definition,
cf.\ \cite[2.2.6.3--2.2.6.7]{lieblich_moduli-of-twisted-sheaves}. We abbreviate
$\Ass(X)=\Ass_X(\sO_X)$.

In particular, if $\map{f}{X}{Y}$ is locally of finite type and $\xi\in |Y|$ is
a point, then we may define $\Ass(X_\xi)\subseteq |f|^{-1}(\xi)$ as the image
of $\Ass(X_y)\to |X|$ for any representative $\map{y}{\Spec k}{Y}$ of $\xi$.
\end{xpar}

\begin{definition}[{\cite[D\'ef.~3.2.2]{raynaud-gruson}}]
Let $\map{f}{X}{Y}$ be a morphism of algebraic stacks that is locally of
finite type. The \emph{relative assassin} $\Ass(X/Y)$ is the subset
$\bigcup_{y\in |Y|} \Ass(X_y)$ of $|X|$.
\end{definition}

Note that $X$ and $Y$ need not be noetherian in the definition above, but
the finiteness condition ensures that the fibers are locally noetherian.
If $f$ is flat and $X$ and $Y$ are locally noetherian, then
$\Ass(X)=\bigcup_{y\in \Ass(Y)} \Ass(X_y) \subseteq \Ass(X/Y)$ by~\eqref{E:Ass}.
The advantage of $\Ass(X/Y)$ is that it behaves well with respect to any base
change $Y'\to Y$, whereas $\Ass(X)$ does not behave well
with respect to non-flat base change, e.g., passage to a fiber.


If $\map{p}{X'}{X}$ is flat and locally of finite type,
then $p(\Ass(X'/Y))\subseteq \Ass(X/Y)$ with equality if $p$ is
surjective; this follows from property (ii) above.

\begin{definition}
Let $\map{f}{X}{Y}$ be a morphism of algebraic stacks. We say that $f$ is
\emph{schematically dominant} if $\sO_Y\to f_*\sO_X$ is injective as a morphism
of lisse-\etale{} sheaves.
\end{definition}

This agrees with the usual definition for schemes~\cite[11.10.2]{egaIV} since
that notion is stable under base change by flat morphisms that are locally of
finite presentation~\cite[11.10.5 (ii) b)]{egaIV}. It follows that our notion
for algebraic stacks also is stable under base change by flat morphisms that
are locally of finite presentation. When $f$ is quasi-compact, the notion is
stable under arbitrary flat base change~\cite[11.10.5 (ii) a)]{egaIV}.

If
$\map{p}{X'}{X}$ is another morphism and $f\circ p$ is schematically dominant,
then so is $f$. If $f$ and $p$ are schematically dominant, then so is $f\circ
p$. In particular, morphisms that are covering in the fppf topology are
schematically dominant.

\begin{definition}
Let $S$ be an algebraic stack and let $\map{f}{X}{Y}$ be a morphism of
algebraic stacks over $S$. We say that $f$ is $S$-universally schematically
dominant if $\map{f'}{X\times_S S'}{Y\times_S S'}$ is schematically dominant
for every morphism $S'\to S$.
\end{definition}

\begin{proposition}\label{P:rel-faithfully-flat}
Let $S$, $X$ and $Y$ be algebraic stacks and let $\map{f}{X}{Y}$ and $Y\to S$
be flat morphisms that are locally of finite presentation.
The following are equivalent.
\begin{enumerate}
\item The morphism $f$ is $S$-universally schematically dominant.
\item The image $f(X)$ contains the relative assassin $\Ass(Y/S)$.
\end{enumerate}
\end{proposition}
\begin{proof}
Since $f$ is open and faithfully flat onto its image, we may assume that $f$ is
an open immersion. As the question is fppf-local on $Y$ and $S$ we may assume
that $Y$ and $S$ are affine schemes. The result is
then~\cite[Prop.~11.10.10]{egaIV} (or~\cite[Cor.~3.2.6]{raynaud-gruson}).
\end{proof}

\begin{definition}\label{D:rel-faithfully-flat}
Let $\map{f}{X}{Y}$ and $\map{g}{Y}{S}$ be morphisms, locally of finite
presentation, between algebraic stacks such that $g$ is flat. We say that
$f$ is \emph{$S$-faithfully flat} if $f$ is flat and the equivalent
conditions of Proposition~\pref{P:rel-faithfully-flat} hold.
\end{definition}

This terminology is explained
by the following lemma.

\begin{lemma}
Let $\map{f}{X}{Y}$ and $\map{\pi}{Y}{S}$ be morphisms of algebraic stacks.
Assume that $f$ is $S$-universally schematically dominant.
Given $\sF\in\QCoh(Y)$ and $\sG\in\QCoh(S)$, we have that
\begin{enumerate}
\item the unit map $\map{\eta_{\pi^*\sG}}{\pi^*\sG}{f_*f^*\pi^*\sG}$ is
  injective; and
\item a morphism $\map{\theta}{\sF}{\pi^*\sG}$ is zero if and only if
  $f^*\theta$ is zero.
\end{enumerate}
\end{lemma}
\begin{proof}
Consider $S'=\Spec(\sO_S\oplus \sG)$, where $\sG$ is square-zero, and let
$X'=X\times_S S'$ and
$Y'=Y\times_S S'$. Then $\map{f'}{X'}{Y'}$ is schematically dominant, that is,
$\sO_Y\oplus \pi^*\sG\to f_*(\sO_X\oplus f^*\pi^*\sG)$ is injective.
It follows that $\eta_{\pi^*\sG}$ is injective.

If $\theta$ is zero, then so is $f^*\theta$. Conversely, if $f^*\theta$ is
zero, then so is $\eta_{\pi^*\sG}\circ \theta=(f_*f^*\theta)\circ \eta_{\sF}$.
It follows that $\theta$ is zero since $\eta_{\pi^*\sG}$ is injective.
\end{proof}

\begin{lemma}\label{L:inclusion-can-be-checked-S-fppf-local}
Let $\map{f}{X}{Y}$ and $\map{\pi}{Y}{S}$ be flat morphisms, that are locally
of finite presentation, between algebraic stacks. Let $\sF_0\subseteq \sF$ be
quasi-coherent $\sO_S$-modules and let $\sG_0\subseteq \pi^*\sF$ be a
quasi-coherent $\sO_Y$-submodule. Assume that $f$ is $S$-faithfully flat. Then
$\sG_0\subseteq \pi^*\sF_0$ if and only if $f^*\sG_0\subseteq f^*\pi^*\sF_0$.
\end{lemma}
\begin{proof}
Let $\sF'=\sF/\sF_0$. Consider the map
$\theta\colon \sG_0\inj \pi^*\sF\surj \pi^*\sF'$.
Then $\sG_0\subseteq \pi^*\sF_0$ if and only if $\theta=0$ and
$f^*\sG_0\subseteq f^*\pi^*\sF_0$ if and only if $f^*\theta=0$. Thus, the
result follows from the previous lemma.
\end{proof}

\end{section}


\begin{section}{Pure morphisms of algebraic stacks}\label{S:pure}
We begin by recalling the definition of pure morphisms of
schemes~\cite[D\'ef.~3.3.3]{raynaud-gruson}.


\begin{definition}
Let $\map{f}{X}{S}$ be a morphism of schemes,
locally of finite type. Let $s\in S$ be a point and let
$\bigl(\widetilde{S},\widetilde{s}\bigr)\to (S,s)$ be the
henselization and $\widetilde{X}=X\times_S \widetilde{S}$. We say that $f$ is
\begin{enumerate}
\item \emph{pure along $X_s$} if for every point $s_1\in \widetilde{S}$, every
associated
point $x_1\in \Ass(\widetilde{X}_{s_1})$ is the generization of a point in $X_s$;
\item \emph{pure} if $f$ is pure along $X_s$ for every $s\in S$; and
\item \emph{universally pure}, if $\map{f'}{X\times_S S'}{S'}$ is
pure for every morphism $S'\to S$.
\end{enumerate}
\end{definition}

\begin{xpar}[Examples]\label{X:purity-examples}
The two key examples of pure morphisms are~\cite[Ex.~I.3.3.4]{raynaud-gruson}:
\begin{enumerate}
\item proper morphisms, and
\item faithfully flat morphisms, locally of finite type,
with fibers that are geometrically irreducible without embedded
components.
\end{enumerate}
\end{xpar}

\begin{xpar}[Base change: descent]\label{X:purity-base-change:descent}
If $S'\to S$ is \emph{faithfully flat} and $f'$ is pure, then $f$ is pure.
Indeed, for every $s'\in |S'|$ with image $s\in |S|$, the morphism between
henselizations $(\widetilde{S}',s')\to (\widetilde{S},s)$ is surjective.  If
$x_1\in \Ass(\widetilde{X}_{s_1})$, then there exists $x'_1\in
\Ass(\widetilde{X}'_{s'_1})$ above $x_1$ (\ref{X:ass}, \itemref{XI:flat}) and,
by
purity, a specialization $x'\in X'_{s'}$. Its image $x\in X_s$, is a
specialization of $x_1$.
\end{xpar}

\begin{xpar}[Base change: stability]\label{X:purity-base-change:flat-is-stable}
If $f$ is \emph{flat}, pure and of finite presentation, then $f$ is
universally pure~\cite[3.3.7]{raynaud-gruson}.
Also, every pure morphism of finite presentation is universally pure when $S$
is locally noetherian~\cite[\spref{05J8}]{stacks-project} but not
for general $S$~\cite[\spref{05JJ}]{stacks-project}.
\end{xpar}

\begin{xpar}[Composition]\label{X:purity-compositions}
Let $\map{f}{X}{Y}$ and $\map{g}{Y}{S}$ be morphisms of schemes, locally of
finite type. If $f$ and $g$ are pure, then $g\circ f$ need not be pure,
e.g., the composition
$\Spec(k[x,y]/xy-1)\inj \Spec k[x,y]\to \Spec k[x]$ is not pure. On the other
hand, if $f$ is flat and pure and $g$ is pure, then $g\circ
f$ is pure. Indeed, the map
$\tilde{f}\colon \widetilde{X}=X\times_S \widetilde{S}\to \widetilde{Y}=Y\times_S \widetilde{S}$ is pure along $X_y$ for every $y\in Y_s$
since the henselization of $Y$ at any point of $Y_s$ factors through
$\widetilde{Y}$. Moreover, since $\tilde{f}$ is flat, we have that
$\tilde{f}(\Ass(\widetilde{X}_{s_1}))\subseteq \Ass(\widetilde{Y}_{s_1})$ for
every $s_1\in \widetilde{S}$.
Also, if $f$ is faithfully flat and $g\circ f$ is pure, then $g$
is pure. Indeed, for every point $s_1\in \widetilde{S}$, we have that
$\tilde{f}(\Ass(\widetilde{X}_{s_1}))= \Ass(\widetilde{Y}_{s_1})$.
\end{xpar}

To extend purity to morphisms of stacks, we give a slightly different definition.

\begin{definition}
Let $\map{f}{X}{S}$ be a morphism between algebraic stacks that is
quasi-separated and
locally of finite type. When $S$ is quasi-separated, we say that $f$ is
\emph{weakly closed} if $f(Z)$ is closed for every closed irreducible subset $Z\subset
|X|$, such that the generic point of $Z$ is associated in its fiber.
We say that $f$ is \emph{universally weakly closed}, if
$\map{f'}{X\times_S S'}{S'}$ is
weakly closed for every morphism $S'\to S$ where $S'$ is quasi-separated.
\end{definition}

The remarks in~\pref{X:purity-examples}, \pref{X:purity-base-change:descent}
and~\pref{X:purity-compositions} hold for ``pure'' replaced by ``weakly
closed''. For Remark~\pref{X:purity-base-change:descent}, note that $f$
is weakly closed if and only if $f(\overline{\{z\}})$ is stable under
specialization for every $z\in \Ass(X/S)$, and this can be checked flat-locally
on $S$.
The analogue of Remark~\pref{X:purity-base-change:flat-is-stable}
is false, which is not surprising: the good notion is universally weakly closed
for which we have the following valuative criterion.

\begin{proposition}
Let $\map{f}{X}{S}$ be a quasi-separated morphism, locally of finite type,
between
algebraic stacks. Then
the following are equivalent:
\begin{enumerate}
\item $f$ is universally weakly closed;
\item for every valuation ring $V$ and morphism
$\Spec V\to S$, the base change $X\times_S \Spec V\to
\Spec V$ is weakly closed; and
\item for every valuation ring $V$, morphism
$\Spec V\to S$, and associated point $z$ in the generic fiber
  $X\times_S \Spec K(V)$,
  the closure of $z$ in $|X\times_S \Spec V|$ surjects onto $\Spec V$.
\end{enumerate}
If $f$ is a morphism of schemes, then this is equivalent to:
\begin{enumerate}
\renewcommand{\theenumi}{{\upshape{(\roman{enumi}$'$)}}}
\item $f$ is universally pure.
\end{enumerate}
\end{proposition}
\begin{proof}
Clearly, (i)$\implies$(ii)$\implies$(iii). If $f$ is a morphism of schemes,
then trivially (i)$\implies$(i$'$) and we note that (i$'$)$\implies$(ii) since it
is enough to verify (ii) for henselian valuation rings.

To see that (iii)$\implies$(i) it is enough
to prove that $f$ is weakly closed. Let $z\in |X|$ be a point that is associated
in its fiber and let $Z=\overline{\{z\}}$. 
It is enough to prove that $f(Z)=\overline{\{f(z)\}}$. This can be
verified after the base change $S'=\Spec V\to S$ for every valuation ring $V$
and every dominant morphism $\Spec V\to \overline{\{f(z)\}}$. Then
$f(Z)=\Spec V$ by (iii) and the result follows.
\end{proof}

\begin{definition}
Let $\map{f}{X}{Y}$ be a flat morphism of finite presentation between
algebraic stacks. We say that $f$ is \emph{pure} if it is universally weakly
closed.
\end{definition}

This definition
coincides with the usual definition for flat morphisms of schemes by~\pref{X:purity-base-change:flat-is-stable}.
It also coincides with the definition of pure
in~\cite[B.1]{romagny_components-in-families}.

The following lemma, which is a direct transcription of
an argument in~\cite[proof of Prop.~3.3.6]{raynaud-gruson}, shows that
a flat morphism $X\to S$ of finite presentation is weakly closed if and only if
the map $\Ass(X/S)\to S$ is closed under specializations, i.e., if subsets
closed under specialization in $\Ass(X/S)$ maps to subsets closed under
specialization in $S$.
\begin{lemma}\label{L:ass-closure}
Let $S$ be a scheme and let $X$ be an algebraic stack that is flat and
of finite presentation over $S$. Let $s,s_1\in |S|$ and $x_1\in \Ass(X_{s_1})$.
If $|X_s|\cap \overline{\{x_1\}}\neq \emptyset$, then
$\Ass(X_s)\cap \overline{\{x_1\}}\neq \emptyset$.
\end{lemma}
\begin{proof}
We may assume that $S=\Spec A$ is affine. Pick a smooth presentation
$\map{p}{U=\Spec B}{X}$. If
${|X_s|\cap \overline{\{x_1\}}\neq \emptyset}$, then there exists a point $u_1\in
U$ above $x_1$ such that $|U_s|\cap \overline{\{u_1\}}\neq \emptyset$. We
may assume that $u_1$ is maximal in $p^{-1}(x_1)$ and then $u_1\in \Ass(U_{s_1})$.
Since $p(\Ass(U_s))=\Ass(X_s)$, it is enough to prove that
$\Ass(U_s)\cap \overline{\{u_1\}}\neq \emptyset$.

Let $u\in |U_s|\cap \overline{\{u_1\}}$ and
let $\Sigma\subseteq \sO_{U,u}$ be the set of elements whose images in
$\sO_{U,u}\otimes \kappa(s)$ are non-zero divisors. Then $\sO_{U,u}\to
\Sigma^{-1}\sO_{U,u}$ is $A$-universally injective and $\Sigma^{-1}\sO_{U,u}$
is a semi-local ring whose maximal ideals are associated points of
$U_s$~\cite[3.2.5]{raynaud-gruson}. In particular, the morphism
$\sO_{U,u}\otimes \kappa(s_1)\to (\Sigma^{-1}\sO_{U,u})\otimes \kappa(s_1)$ is
injective. Since $u_1$ is associated in $\Spec(\sO_{U,u}\otimes \kappa(s_1))$,
this means that $u_1\in \Spec(\Sigma^{-1}\sO_{U,u}\otimes \kappa(s_1))$; hence
$u_1$ is a generization of an associated point $u_0$ of $U_s$.
\end{proof}

\end{section}


\begin{section}{Homological projectivity}\label{S:hom-proj}
The main theorem of~\cite[\S I.3]{raynaud-gruson} is the following relation
between purity and projectivity for affine morphisms.

\begin{theorem}[Raynaud--Gruson]\label{T:RG-main}
Let $\map{f}{X}{Y}$ be an affine finitely presented morphism of schemes.
The following
are equivalent:
\begin{enumerate}
\item $f$ is flat and pure;
\item $f_*\sO_X$ is locally projective; and
\item $f_*\sO_X$ is locally free.
\end{enumerate}
\end{theorem}
\begin{proof}
The equivalence between (i) and (ii) is~\cite[Thm.~I.3.3.5]{raynaud-gruson}.
The equivalence between (ii) and (iii) is~\cite[Cor.~I.3.3.12]{raynaud-gruson}.
Note that if $Y$ is noetherian, then the latter equivalence follows directly
from
Proposition~\pref{P:loc-proj-vs-loc-free}~\itemref{PI:loc-free=loc-proj:noeth}. The
non-noetherian case follows from the noetherian case using the equivalence
between (i) and (ii) and using that pure morphisms behave well under
approximation~\cite[Cor.~I.3.3.10]{raynaud-gruson}.
\end{proof}

Local projectivity of $f_*\sO_X$ is not local on $X$. To obtain a non-affine
analogue of the theorem above, we introduce the following definition.


\begin{definition}
Let $\map{f}{X}{Y}$ be a flat morphism of finite presentation between
algebraic stacks. We say that $f$ is \emph{homologically projective}
(resp.\ \emph{strongly homologically projective}) if there exists
\begin{enumerate}
\item an fppf-covering $\{\Spec(A_i)\to Y\}$; and 
\item flat morphisms $\map{q_i}{\Spec(B_i)}{X\times_Y \Spec(A_i)}$,
  locally of finite presentation;
\end{enumerate}
such that for every $i$
\begin{enumerate}\renewcommand{\theenumi}{{\upshape{(\alph{enumi})}}}
\item the composition
$\Spec(B_i)\xrightarrow{q_i} X\times_Y \Spec(A_i)\to \Spec(A_i)$
makes $B_i$ into a projective $A_i$-module; and
\item $q_i$ is $\Spec(A_i)$-faithfully flat (resp.\ faithfully flat), cf.\ 
  Definition~\pref{D:rel-faithfully-flat}.
\end{enumerate}
\end{definition}
Here ``homological'' is to indicate that projective is interpreted as
in homological algebra and not as in algebraic geometry.
It should not be confused with the notion of
\emph{cohomologically projective} morphisms in~\cite[3.18]{alper_good-mod-spaces}.

By definition, the notion of (strong) homological projectivity is stable under
base
change and fppf-local on the target.
If $\map{p}{X'}{X}$ is faithfully flat and locally of finite presentation and
$f\circ p$ is (strongly) homologically projective, then $f$ is (strongly)
homologically projective
but the converse does not hold. It is, a priori, not clear whether the
composition of
two (strongly) homologically projective morphisms is (strongly) homologically
projective.

Recall that $X$ has the resolution property if every quasi-coherent
sheaf of finite type on $X$ admits a surjection from a vector bundle.
\begin{theorem}\label{T:pure=hom-proj}
Let $\map{f}{X}{Y}$ be a morphism of algebraic stacks that is
flat and of finite presentation. Consider the following conditions:
\begin{enumerate}
\item $f$ is affine and $f_*\sO_X$ is locally projective;
\item $f$ is strongly homologically projective;
\item $f$ is homologically projective; and
\item $f$ is pure.
\end{enumerate}
Then (i)$\implies$(ii)$\implies$(iii)$\iff$(iv). If $f$ is affine, then
all four conditions are equivalent. If $X$ has the resolution property
fppf-locally on $Y$ (e.g., if $f$ is quasi-affine),
then (ii)$\iff$(iii).
\end{theorem}
\begin{proof}
From the definitions, it follows that (i)$\implies$(ii)$\implies$(iii). To
prove that (iii)$\implies$(iv), we may assume that $Y=\Spec A$ and that there
is a $Y$-faithfully flat and finitely presented morphism $U=\Spec B\to X$ such
that $B$ is a projective $A$-module. By Theorem~\pref{T:RG-main}, we have that
$U\to Y$ is pure. Since the image of $U$ contains $\Ass(X/Y)$, it follows that
$X\to Y$ is pure.
When $f$ is affine, (iv)$\implies$(i) by Theorem~\pref{T:RG-main}.

For (iv)$\implies$(iii), suppose that $f$ is pure. As before we may assume that
$Y$ is
affine. Pick a smooth presentation $U=\Spec B\to X$. Let $y\in Y$ be a point.
Then, by~\cite[Prop.\ 3.3.2]{raynaud-gruson}, there exists a commutative diagram
\[
\xymatrix{%
U'\ar[d]\ar[r] & U\ar[d] \\
Y'\ar[r] & Y}
\]
and a point $y'\in Y'$ above $y$ such that
\begin{itemize}
\item $U'\to U$ and $Y'\to Y$ are \etale{}, and $\kappa(y)=\kappa(y')$;
\item $U'=\Spec B'$ and $Y'=\Spec A'$ are affine and $B'$ is a projective
  $A'$-module; and
\item the image of $U'\to U$ contains $\Ass(U_y)$.
\end{itemize}
In particular, the image of $U'\to U\to X$ contains $\Ass(X_y)$. After replacing
$X$, $Y$ and $U$ by their pull-backs along the base change $Y'\to Y$, we may
assume that $Y'=Y$.

We now claim
that the image of $U'\to U\to X$ contains $\Ass(X_{y_1})$ for every generization
$y_1$ of $y$. To see this, let $x_1\in \Ass(X_{y_1})$. By the definition of
purity, there exists a point $x\in X_y\cap \overline{\{x_1\}}$. By
Lemma~\pref{L:ass-closure}, there exists a point
$x_0\in \Ass(X_y)\cap \overline{\{x_1\}}$.
Since $x_0$ is in the image of $U'$, so is its
generization~$x_1$.

By~\cite[Lem.~3.3.9]{raynaud-gruson}, there is then an open neighborhood $y\in
V\subseteq Y$ such that the image of $U'\to U\to X$ contains $\Ass(X_{y_1})$ for
every $y_1\in V$. This means that $U'\to U\to X$ is $Y$-faithfully flat
over $V$, that is, $X\to Y$ is homologically projective over $V$. As the
question is local on $Y$, it follows that $X\to Y$ is homologically
projective.

Under the additional assumption on $X$, we will we prove that
(iv)$\implies$(ii). For this,
we may work locally on $Y$ and assume that $Y=\Spec A$
is affine and that $X$ has the resolution property. Then $X=[U/\GL_n]$ for some
quasi-affine scheme
$U$~\cite{totaro_resolution-property,gross_tensor-generators}. By Jouanolou's
trick, there is an affine vector bundle torsor $E\to
U$~\cite[Lem.~1.5]{jouanolou_Mayer-Vietoris} (also
see~\cite[4.3--4.4]{weibel_homotopy-algebraic-K-theory}). Since $E\to X$ is
flat with geometrically integral fibers, hence flat and pure, it follows
that $E\to X\to Y$ is pure~\eqref{X:purity-compositions}.
Since $E=\Spec B$ is affine, we have that $B$ is $A$-projective; thus, 
$f$ is strongly homologically projective.
\end{proof}



\end{section}


\begin{section}{Existence of minimal subsheaves}\label{S:minimal}
Let $\map{f}{X}{Y}$ be a faithfully flat morphism between quasi-compact
algebraic stacks and let $\sF\in \QCoh(Y)$.
Assume that $\sG_0\subseteq \sG:=f^*\sF$ is a quasi-coherent subsheaf of finite
type. If $\sF$ is the union of its quasi-coherent subsheaves $\sF_\lambda$
of finite type, then, for sufficiently large $\lambda$, we have that
$\sG_0\subseteq f^*\sF_\lambda$.

Conversely, if $\sG=f^*\sF$ is the union of its quasi-coherent subsheaves
$\sG_\lambda$ of finite type and for every $\sG_\lambda$ there exists
$\sF_\lambda\subseteq \sF$ of finite type such that $\sG_\lambda\subseteq
f^*\sF_\lambda$, then $\sF$ is the union of its subsheaves of finite
type.

We will see that, under suitable hypotheses, for every $\sG_\lambda$ of
finite type as above there is a \emph{minimal} $\sF_\lambda$ as above and it is
of finite type. This is, however, not always the case:

\begin{example}\label{E:flat-non-projective}
Let $A$ be a discrete valuation ring with fraction field $K$ and uniformizing
parameter $t$. Let $B=A\times K$, which is a faithfully flat $A$-algebra. Let
$M=A$ and consider the submodule $N_0=(0\times K)\subseteq M\otimes_A B=B$. For
every non-trivial ideal $M_n=(t^n)\subseteq A=M$, we then have that
$N_0\subseteq M_n\otimes_A B=(t^n)\times K$. But the intersection is
$\bigcap M_n=0$ and $N_0\nsubseteq (\bigcap M_n)\otimes_A B=0$. Hence, there is
no minimal submodule $M'$ of $M$ such that $N_0\subseteq M'\otimes_A B$.
\end{example}

The problem in Example~\pref{E:flat-non-projective} is that infinite
intersections do not commute with
flat pull-back. This does not happen if we replace flatness with projectivity.

\begin{lemma}[Serre]\label{L:affine-projective-submodule-descent}
Let $A$ be a ring and let $B$ be an $A$-algebra which is
projective as an $A$-module. Let $M$ be an $A$-module and let $N_0\subseteq
M\otimes_A B$ be a $B$-submodule. Then there is a unique minimal $A$-submodule
$M_0\subseteq M$ such that $N_0\subseteq M_0\otimes_A B$. Moreover,
\begin{enumerate}
\item if $N_0$ is of finite type, then so is $M_0$; and
\item if $A'$ is an $A$-algebra and we let $B'=B\otimes_A A'$, $M'=M\otimes_A
  A'$, $M'_0:=\image(M_0\otimes_A A'\to M')$ and $N'_0=\image(N_0\otimes_B
  B'\to M'\otimes_{A'} B')$, then $M'_0$ is the minimal $A'$-submodule of
  $M'$ such that $N'_0\subseteq M'_0\otimes_{A'} B'$.
\end{enumerate}
\end{lemma}
\begin{proof}
Choose a free $A$-module $F$ such that $B$ is a direct summand of $F$ and pick
a basis $\{e_i\}$ of $F$. Let $M_0\subseteq M$ be an $A$-submodule.
Then $M_0\otimes_A B\subseteq
M_0\otimes_A F\subseteq M\otimes_A F$ and $M\otimes_A B\subseteq M\otimes_A
F$. Let $x\in N_0$ be an element. Then $x=\sum_i x_i \otimes e_i$ in
$M\otimes_A F$, and, using the retraction $F\to B$, we may also write $x=\sum_i
x_i \otimes b_i$ in $M\otimes_A B$. Thus $x\in M_0\otimes_A B$ if and only if
$x_i\in M_0$ for every $i$. It follows that the minimal submodule $M_0$ is the
submodule generated by the $x_i$'s when $x$ ranges over a set of generators of
$N_0$. The remaining claims follows immediately from the construction of~$M_0$.
\end{proof}

Using purity, we give the following global version.

\begin{theorem}\label{T:existence-of-minimal}
Let $\map{f}{X}{Y}$ be a flat morphism of finite presentation between algebraic
stacks. Assume that $f$ is pure. Let $\sF\in\QCoh(Y)$ and let
$\sG_0\subseteq \sG:=f^*\sF$ be a quasi-coherent submodule. Then there is a
unique minimal quasi-coherent submodule $\sF_0\subseteq \sF$ such that
$\sG_0\subseteq
f^*\sF_0$. Moreover,
\begin{enumerate}
\item if $\sG_0$ is of finite type, then so is $\sF_0$; and
\item if $\map{f'}{X'}{Y'}$ is the base change of $f$ along
  a morphism ${\map{g}{Y'}{Y}}$, then the image $\sF'_0$ of
  $g^*\sF_0\to g^*\sF$ is the minimal quasi-coherent submodule
  such that $f'^*\sF'_0$ contains the image of ${g'^*\sG_0\to g'^*\sG}$.
\end{enumerate}
\end{theorem}
\begin{proof}
By Theorem~\pref{T:pure=hom-proj}, $f$ is homologically projective.
By fppf descent, it is enough to prove the statement after replacing $Y$ with
an fppf cover. We may thus assume that $Y=\Spec A$ and that there exists a
$Y$-faithfully flat morphism $\map{q}{X'=\Spec B}{X}$ of finite presentation
such that $B$ is a projective $A$-module.
If $\sF_0\subseteq \sF$ is a submodule, then $\sG_0\subseteq f^*\sF_0$ if and
only if $q^*\sG_0\subseteq q^*f^*\sF_0$ (Lemma~\ref{L:inclusion-can-be-checked-S-fppf-local}).
We may thus replace $X$ with $X'$ and
assume that $X$ and $Y$ are affine. The theorem is then
Lemma~\pref{L:affine-projective-submodule-descent}.
\end{proof}
\end{section}


\begin{section}{Approximation of quasi-coherent sheaves}\label{S:approx}
Let $X$ be a quasi-compact and quasi-separated algebraic stack. We recall
that $X$ has the \emph{completeness
  property} if every quasi-coherent $\sO_X$-module is a directed colimit of
finitely presented $\sO_X$-modules and that $X$ has the \emph{partial
  completeness property} if every quasi-coherent $\sO_X$-module is the union
of its finitely generated quasi-coherent submodules. In the terminology
of~\cite[\S4]{rydh_noetherian-approx}, these two conditions are the conditions
(C1) and (C2) for the category $\QCoh(X)$ and they imply the corresponding
facts for quasi-coherent $\sO_X$-algebras.

We also make the following definition that
extends~\cite[Def.~4.7]{rydh_noetherian-approx}.

\begin{definition}\label{D:semi-noetherian}
An algebraic stack $X$ is \emph{semi-noetherian}
(resp.\ \emph{pseudo-noetherian}) if it is quasi-compact, quasi-separated and
$X'$ has the partial completeness property (resp.\ completeness property) for
every finitely presented morphism $X'\to X$ of algebraic stacks.
\end{definition}

Every pseudo-noetherian algebraic stack is semi-noetherian.
Noetherian algebraic stacks,
quasi-compact and quasi-separated schemes, algebraic spaces and
Deligne--Mumford stacks are examples of
pseudo-noetherian algebraic stacks~\cite[Thm.~A]{rydh_noetherian-approx}.


\begin{proposition}\label{P:semi-noeth:pure}
Let $\map{f}{X}{Y}$ be a faithfully flat and pure morphism of finite
presentation between
quasi-compact
and quasi-separated algebraic stacks. If $X$ has the partial completeness
property, then so has $Y$. In particular, $X$ is semi-noetherian if and only
if $Y$ is semi-noetherian.
\end{proposition}
\begin{proof}
Let $\sF\in\QCoh(Y)$ and write $f^*\sF$ as a union $\bigcup \sG_\lambda$
of quasi-coherent submodules of finite type.
By Theorem~\pref{T:existence-of-minimal}, for every $\lambda$ there exists
a minimal quasi-coherent
subsheaf $\sF_\lambda\subseteq \sF$ of finite type such that
$\sG_\lambda\subseteq f^*\sF_\lambda$. If we let
$\sF'=\bigcup \sF_\lambda\subseteq \sF$, then $f^*\sF'$ contains every
$\sG_\lambda$. It follows that $f^*\sF'=f^*\sF$ and thus $\sF'=\sF$ since $f$
is faithfully flat. We conclude that $Y$ has the partial completeness property.
\end{proof}

\begin{proposition}\label{P:semi-noeth:et}
Let $X$ be an algebraic stack and let $\map{p}{X'}{X}$
be \etale{}, representable, surjective and of finite presentation. Then
$X$ is semi-noetherian if and only if $X'$ is semi-noetherian.
\end{proposition}
\begin{proof}
This is proven exactly as~\cite[Prop.~4.11]{rydh_noetherian-approx}:
\etale{} \devissage{}~\cite[Thm.~D]{rydh_etale-devissage} is used to reduce the
question to where $p$ is either finite, surjective and \etale{} or an \etale{}
neighborhood. These two cases follow from simplified versions of \cite[Lem.~4.9
  and 4.10]{rydh_noetherian-approx} where ``completeness property'' is replaced
with ``partial completeness property''.
\end{proof}

The main theorem will follow from the previous two propositions together
with the following factorization result. For our main theorem we will only
apply it to a smooth and representable morphism (the presentation of a stack).

\begin{theorem}[{\cite[6.8]{laumon},\cite{romagny_components-in-families}}]\label{T:connected-factorization}
Let $\map{f}{X}{Y}$ be a faithfully flat morphism of finite
presentation with geometrically reduced fibers (e.g., $f$ smooth) between
algebraic stacks.
Then there exists an open substack $U\subseteq X$ and a factorization
$f|_U=h\circ g$ such that
\begin{enumerate}
\item $g$ and $h$ are faithfully flat of finite presentation;
\item $h$ is representable and \etale; and
\item $g$ has geometrically integral fibers.
\end{enumerate}
In particular, $g$ is pure. If $f$ is smooth, then $g$ is smooth and
we can take $U=X$.
\end{theorem}
\begin{proof}
First assume that $f$ is smooth. Consider the connected
factorization $X\to \pi_0(X/Y)\to Y$, which is described for morphisms of
schemes
in~\cite[6.8]{laumon} and for an algebraic stack over an algebraic space
in~\cite[Thm.~2.5.2]{romagny_components-in-families}. Since the construction
commutes with base change, it generalizes to our situation as well.
In this factorization
$\map{g}{X}{\pi_0(X/Y)}$ is smooth with geometrically connected fibers and
$\map{h}{\pi_0(X/Y)}{Y}$ is \etale{}, representable and of finite presentation,
but not necessarily
separated~\cite[Thm.~2.5.2 (i), (ii)]{romagny_components-in-families}.

In the general case we use the functor of irreducible components of
Romagny. The
\emph{unicomponent locus} $U\subseteq X$ is the subset of points that belong to
exactly one irreducible component of their fibers. It is open and quasi-compact
and there
is a factorization $U\to \Irr(X/Y)\to Y$ where the first morphism has
geometrically integral fibers and the second is surjective,
\etale{},
representable and of finite
presentation~\cite[Thm.~2.5.2 (i), (iii)]{romagny_components-in-families}.
\end{proof}

We now obtain the following equivalent form of the main theorem.

\begin{theorem}\label{T:main-theorem-variant}
Let $X$ be a quasi-compact and quasi-separated algebraic stack. Then $X$
is semi-noetherian.
\end{theorem}
\begin{proof}
Pick a smooth presentation $\Spec B\to X$. Theorem~\pref{T:connected-factorization} gives a factorization $\Spec B\to W\to X$ where $\Spec B\to W$ is smooth,
surjective and pure and $W\to X$ is \etale, surjective and of finite
presentation. The result now follows from
Propositions~\pref{P:semi-noeth:pure} and~\pref{P:semi-noeth:et}.
\end{proof}

\begin{remark}\label{R:conjectures}
To answer Conjectures~\tref{CONJ:completeness}
and~\tref{CONJ:approximation}, we may argue as in the proof of
Theorem~\pref{T:main-theorem-variant} using~\cite[Prop.~4.11
  and Lem.~7.9]{rydh_noetherian-approx}. This reduces the situation to where
$X$ has a smooth presentation $U\to X$ with geometrically connected fibers.
The author
hopes that the purity of $U\to X$ and its characterization as
homological projectivity can be used to settle the conjectures.
\end{remark}

\end{section}


\begin{section}{Applications}
We conclude with some applications of the main theorem.
\begin{theorem}[Zariski's main theorem]
Let $\map{f}{X}{Y}$ be a morphism between quasi-compact and quasi-separated
algebraic stacks. Then the following are equivalent:
\begin{enumerate}
\item $f$ is representable, separated and quasi-finite; and
\item there is a factorization $f=\overline{f}\circ j$ where $j$
  is a quasi-compact open immersion and $\overline{f}$ is finite.
\end{enumerate}
\end{theorem}
\begin{proof}
This follows from~\cite[Thm.~16.5 (ii)]{laumon} and the main theorem (taking
into account that the finite presentation assumption of \loccit\ can be
avoided by replacing the reference to [EGA] IV 8.12.6 with [EGA] IV 18.12.13).
An essentially identical proof is given
in~\cite[Thm.~8.6 (ii)]{rydh_noetherian-approx} (use the
partial completeness property instead of the completeness property).
\end{proof}

\begin{proposition}\label{P:qcopen-has-fpclosed-complement}
Let $X$ be a quasi-compact and quasi-separated algebraic stack and let
$U\subseteq X$ be a quasi-compact open substack. Then there exists a
closed immersion $Z\inj X$ of finite presentation such that $U=X\smallsetminus
Z$.
\end{proposition}
\begin{proof}
Let $I\subseteq \sO_X$ be the quasi-coherent sheaf of ideals defining
$Z_\red=(X\smallsetminus U)_\red$. Write $I=\bigcup I_\lambda$ as a union of
quasi-coherent ideals of finite type. If $Z_\lambda$ denotes the finitely
presented closed substack corresponding to $I_\lambda$, then $\cap
Z_\lambda=Z_\red$. Since $U$ is quasi-compact it follows that
$|Z_\lambda|=|Z_\red|$ for all sufficiently large $\lambda$. We may take
$Z=Z_\lambda$ for any such $\lambda$.
\end{proof}

As a third application, we have the existence of flattening stratifications
for finitely presented morphisms.

\begin{theorem}\label{T:flattening-stratification}
Let $X$ be a quasi-compact and quasi-separated algebraic stack and let $W\to X$
be a morphism of finite presentation. Then there exists a sequence of finitely
presented closed substacks $\emptyset=X_0\inj X_1\inj\dots \inj X_n$ such that
$|X_n|=|X|$ and the restriction of $W\to X$ to $X_k\smallsetminus X_{k-1}$ is
flat
for every $k=1,2,\dots,n$.
\end{theorem}
\begin{proof}
The result is well-known when $X$ is noetherian: let $X_n=X_\red$; pick a
smooth presentation $\map{p}{\Spec(A)}{X_n}$; choose a non-empty open subscheme
$V\subseteq \Spec(A)$ over which $W$ is flat (generic flatness); let
$X_{n-1}=(X\smallsetminus p(V))_\red$. The result now follows by noetherian
induction.

If $X$ is affine, the result follows by standard limit methods: there is a
noetherian affine scheme $X_0$, a morphism $X\to X_0$ and a
morphism $W_0\to X_0$ of finite presentation that pull-backs to $W\to X$.
The pull-back of a solution
to the problem for $W_0\to X_0$ gives a solution for $W\to X$.

In the general case, we pick a smooth presentation $\map{p}{X'=\Spec(A)}{X}$ and
choose a filtration $X'_0\inj X'_1\inj\dots\inj X'_n$ that solves the problem
over~$X'$. We will prove that $X$ has a filtration of length $n$ that solves the
problem. Set-theoretically, we will have $|X_k|=X\smallsetminus
p(X'\smallsetminus X'_k)$.
If $n=0$, the problem is trivial. By induction on $n$, we may assume that there
exists a filtration of length $n-1$ on every closed substack $Q\inj X$ such
that $|p^{-1}(Q)|\subseteq |X'_{n-1}|$.

The subset $p(X'\smallsetminus X'_{n-1})$ is open and quasi-compact, hence
there is a finitely presented closed substack $Z\inj X$ such that
$X\smallsetminus Z=p(X'\smallsetminus X'_{n-1})$
(Proposition~\ref{P:qcopen-has-fpclosed-complement}).

Since $p$ is smooth, we have that $p^{-1}(X_\red)=X'_\red$ and hence
$p^{-1}(X_\red)\inj X'$ factors through $X'_n$. Writing the nilradical of
$\sO_X$ as a union of quasi-coherent ideals of finite type, we may write the
nil-immersion $X_\red\inj X$ as an intersection of finitely presented
nil-immersions $X_\lambda\inj X$. For sufficiently large $\lambda$, we have
that $p^{-1}(X_\lambda)\inj X'$ factors through $X'_n$. Then $W\to X$ is flat
over $X_\lambda\smallsetminus Z$ for such $\lambda$ since
$p^{-1}(X_\lambda)\smallsetminus X'_{n-1}\to X_\lambda\smallsetminus Z$ is
smooth and
surjective.

We let $X_n=X_\lambda$ and $Q=Z\cap X_\lambda$. Then, by induction there is a
filtration $X_0\inj X_1\inj\dots\inj X_{n-1}\inj Q$ with $|X_{n-1}|=|Q|$ such
that $W\to X$ is flat over the strata. The result follows.
\end{proof}

As a fourth application, we have the existence of stratifications into gerbes
for stacks with finitely presented inertia.

\begin{corollary}\label{C:stratification-into-gerbes}
Let $X$ be a quasi-compact and quasi-separated algebraic stack with inertia
of finite presentation. Then there exists a sequence of finitely
presented closed substacks $\emptyset=X_0\inj X_1\inj\dots \inj X_n$ such that
$|X_n|=|X|$ and $X_k\smallsetminus X_{k-1}$ is an fppf gerbe over an affine
scheme
for every
$k=1,2,\dots,n$.
\end{corollary}
\begin{proof}
Apply Theorem~\pref{T:flattening-stratification} on $I_X\to X$ to obtain a
stratification into fppf gerbes over quasi-compact and quasi-separated
algebraic spaces. By Proposition~\pref{P:qcopen-has-fpclosed-complement},
it remains to prove that a quasi-compact and quasi-separated
algebraic space $S$ can be stratified into affine schemes. Pick an
approximation $S\to S_0\to \Spec \Z$, that is, an algebraic space $S_0$
of finite presentation over $\Spec \Z$ and an affine morphism
$S\to S_0$~\cite[Thm.~D]{rydh_noetherian-approx}.
It is enough to stratify $S_0$ into affine schemes. This can be done
by noetherian induction since $S_0$ has an open subspace that is a scheme.
\end{proof}

For a general quasi-compact and quasi-separated algebraic stack, the inertia is
only of finite type. In this case, it is not always possible to find finitely
presented stratifications as in
Corollary~\pref{C:stratification-into-gerbes}. In fact, sometimes even an
infinite number of strata is required~\cite[\spref{06RE}]{stacks-project}.

As a final application, we see that two different definitions of projectivity
and quasi-projectivity over algebraic stacks are equivalent. Our main
definition is analogous to that for schemes in EGA~\cite[D\'efs.\ 5.3.1 and
  5.5.2]{egaII}.
\begin{definition}
A representable morphism $\map{f}{X}{Y}$ of algebraic stacks is
\begin{enumerate}
\item \emph{quasi-projective} if $f$ is of finite type and there exists an
  $f$-ample invertible $\sO_X$-module; and
\item \emph{projective} if $X$ is $Y$-isomorphic to a closed substack
  of a projective bundle $\P_Y(\sE)$ where $\sE$ is a quasi-coherent
  $\sO_Y$-module of finite type.
\end{enumerate}
\end{definition}
Note that being $f$-ample is an fppf-local property on the
target~\cite[Cor.~2.7.2]{egaIV} and hence makes sense for representable
morphisms. Similarly, projective bundles is a local construction on the base.

\begin{theorem}[cf.\ {\cite[Prop.~5.3.2 \& Thm.~5.5.3]{egaII}}]
Let $Y$ be a quasi-compact and quasi-separated algebraic stack and let
$\map{f}{X}{Y}$ be a representable morphism. Then
\begin{enumerate}
\item\label{TI:quasi-proj}
  $f$ is quasi-projective if and only if there exists a quasi-compact
  immersion $X\inj \P_Y(\sE)$ over $Y$, where $\sE$ is a quasi-coherent
  $\sO_Y$-module of finite type.
\item\label{TI:proj}
  $f$ is projective if and only if it is proper and quasi-projective.
\end{enumerate}
\end{theorem}
\begin{proof}
If $\map{i}{X}{\P_Y(\sE)}$ is a quasi-compact immersion, then
$i^*\sO_{\P(\sE)}(1)$ is very ample and $f$ is quasi-projective. Conversely,
assume that $f$ is quasi-projective and let $\sL$ be an $f$-ample invertible
sheaf. There is a natural map $\map{\sigma}{f^*f_*\sL}{\sL}$ and when this map
is surjective, we have an induced morphism
$\map{r_{\sL,\sigma}}{X}{\P(f_*\sL)}$.  Choose a presentation
$\map{g}{Y'}{Y}$. After replacing $\sL$ with a sufficiently large power,
the invertible sheaf
$g'^*(\sL)$ becomes very ample which implies that $\sigma$ is surjective and
$r_{\sL,\sigma}$ is an immersion~\cite[Prop.~4.4.4]{egaII}. Write $f_*\sL$ as
the union of its finitely generated submodules $\sE_\lambda$. Then for
sufficiently large $\lambda$, the map
$\map{\sigma_\lambda}{f^*\sE_\lambda}{\sL}$ is surjective and the induced
morphism $\map{r_{\sL,\sigma_\lambda}}{X}{\P(\sE_\lambda)}$ is an
immersion~\cite[pf.\ of Prop.~4.4.1 (ii)]{egaII}.

If $f$ is projective, then $f$ is quasi-projective (as before) and proper
(check locally on $Y$). Conversely, if $f$ is quasi-projective and proper,
then by~\itemref{TI:quasi-proj}, there is an immersion $X\inj \P_Y(\sE)$
which is closed since $f$ is proper.
\end{proof}

\end{section}


\bibliography{noetherian-qaff}
\bibliographystyle{dary}

\end{document}